\documentclass[12pt,reqno]{amsart}

\usepackage{verbatim, amssymb, enumitem, mathtools}

\usepackage[breaklinks=true,colorlinks=true,linkcolor=blue,citecolor=red,urlcolor=blue,psdextra,pdfencoding=auto]{hyperref}
\allowdisplaybreaks

\usepackage{orcidlink}

\renewcommand \a{\alpha}
\renewcommand \b{\beta}

\newcommand \la{\lambda}
\newcommand \ve{\varepsilon}

\newcommand \br{\mathbb{R}}
\newcommand \bc{\mathbb{C}}
\newcommand \bh{\mathbb{H}}
\newcommand \bo{\mathbb{O}}

\newcommand \rk{\operatorname{rk}}

\newcommand \Span{\operatorname{Span}}
\newcommand \Tr{\operatorname{Tr}}

\newcommand \SO{\mathrm{SO}}
\newcommand \Sp{\mathrm{Sp}}
\newcommand \SU{\mathrm{SU}}
\newcommand \SL{\mathrm{SL}}
\newcommand \Ff{\mathrm{F}_4}
\newcommand \Spin{\mathrm{Spin}}

\newcommand \cI{\mathcal{I}}

\newcommand \cK{\mathcal{K}}

\newcommand \sK{\mathsf{K}}
\newcommand \sS{\mathsf{S}}

\newcommand \sT{\mathsf{T}}

\newcommand\ag{\mathfrak a}

\newcommand\g{\mathfrak g}

\newcommand\h{\mathfrak h}

\newcommand\m{\mathfrak m}
\newcommand \so{\mathfrak{so}}

\newcommand \gl{\mathfrak{gl}}

\newcommand \s{\mathfrak{s}}

\newcommand \<{\langle}
\renewcommand \>{\rangle}
\newcommand \ip{\<\cdot,\cdot\>}

\newcommand \GL{\mathrm{GL}}

\newtheorem{theorem}{Theorem}
\newtheorem*{theorem*}{Theorem}

\newtheorem*{corollary*}{Corollary}
\newtheorem*{conj*}{Conjecture}
\newtheorem{lemma}{Lemma}
\newtheorem{proposition}{Proposition}
\newtheorem*{prop*}{Proposition}

\theoremstyle{definition}

\newtheorem*{definition*}{Definition}

\theoremstyle{remark}

\newtheorem*{notation*}{Notation}
\newtheorem*{algorithm*}{Algorithm}
\newtheorem*{example*}{Example}

\begin{document}

\title{Quadratic Killing tensors on some symmetric spaces of higher rank} 

\author[An Ky Nguyen]{An Ky Nguyen \orcidlink{0009-0001-8929-7060}}
\address{Department of Mathematical and Physical Sciences, La Trobe University, Melbourne, Victoria, 3086, Australia}
\email{AnKy.Nguyen@latrobe.edu.au}

\author[Yuri Nikolayevsky]{Yuri Nikolayevsky \orcidlink{0000-0002-9528-1882}}
\address{Department of Mathematical and Physical Sciences, La Trobe University, Melbourne, Victoria, 3086, Australia}
\email{y.nikolayevsky@latrobe.edu.au}

\thanks {The second named author was partially supported by ARC Discovery grant DP210100951.}



\begin{abstract}
All Killing tensor fields on the spaces of constant curvature and on the complex projective space are decomposable, that is, can be represented as the sum of symmetric tensor products of Killing vector fields (equivalently, every polynomial integral of the geodesic flow is a polynomial in the linear integrals). This is no longer true for quadratic Killing tensor fields on the quaternionic projective spaces $\bh P^n, \, n \ge 3$, and on the Cayley projective plane $\bo P^2$. We prove that for the real Grassmannians and for the spaces $\SL(n)/\SO(n)$, all quadratic Killing tensor fields are decomposable.
\end{abstract}

\maketitle

\section{Introduction}
\label{s:intro}

A \emph{Killing tensor field} $\cK=\cK(x)_{i_1\dots i_d}$ of rank $d \ge 1$ on a Riemannian manifold $(M,ds^2=g_{ij}\,dx^i dx^j)$ is a covariant symmetric tensor field that satisfies the Killing equation
\begin{equation}\label{eq:defK}
  \cK_{(i_1\dots i_d,j)}=0,
\end{equation}
where the comma denotes the covariant derivative and the parentheses denote the symmetrisation by all indices. This definition is equivalent to the fact that the function $\xi\in T_xM \mapsto \cK(x)_{i_1\dots i_d} \xi^{i_1} \cdots \xi^{i_d}$ polynomial in the velocities is an integral of the geodesic flow of $(M,ds^2)$: for any naturally parameterised geodesic $s \mapsto \gamma(s)$ of $(M,ds^2)$, the function $s \mapsto \cK(\gamma(s))_{i_1\dots i_d} (\dot{\gamma}(s))^{i_1} \dots (\dot{\gamma}(s))^{i_d}$ is constant.

Killing tensor fields of rank $d=1$ are called \emph{Killing \emph{(}co\emph{)}vector fields}. It is well known that a vector field is Killing if and only if the $1$-parametric group of diffeomorphisms of $M$ which it generates is a group of isometries. 

The space of all Killing tensor fields on $(M,ds^2)$ forms an associative, commutative, graded algebra $\sK(M)$ relative to the symmetric tensor product. It contains a subalgebra $\sS(M)$ generated by Killing vector fields. The elements of $\sS(M)$ are called \emph{decomposable}.

In general, a Killing tensor field of rank $d \ge 2$ on a Riemannian manifold does not have to be decomposable, even if one disregards the polynomials in the metric tensor. Historically, the first example is due to Darboux: a $2$-dimensional Liouville metric $ds^2=(\la(x)+\mu(y)) (dx^2 + dy^2)$ admits quadratic Killing tensors not proportional to the metric, but in general, has trivial isometry group. On the other hand, any Killing tensor field on the space of constant curvature is decomposable~\cite{Tho,ST1,Tak}. Motivated by these results and by the fact that symmetric spaces have a ``large" isometry group, the following question has been suggested in \cite[Question~3.9]{BMMT}: ``In a symmetric space, is every Killing tensor field decomposable?''

The answer to this question is in the positive for the complex projective spaces~\cite[Corollary~5]{East}, \cite[Theorem~2.2]{ST2}, but is in the negative already for quadratic Killing tensor fields on almost all other compact rank-one symmetric spaces: on the quaternionic projective spaces $\bh P^n=\Sp(n+1)/(\Sp(n)\Sp(1))$ with $n \ge 3$, and on the Cayley projective plane $\bo P^2 = \Ff/\Spin(9)$ \cite{MN1}.

At present, the quadratic Killing tensor fields on \emph{rank one} symmetric spaces are fully understood: for the compact ones this is done in~\cite[Theorem~4]{MN2}, and the passage to the noncompact duals follows from~\cite[Theorem~2]{MN2}. Moreover, we know that the algebra of Killing tensors (of \emph{all ranks}) on the quaternionic projective space is generated by the linear and the quadratic Killing tensors~\cite[Theorem~1]{DN}.

For symmetric spaces of \emph{higher rank}, only the quadratic Killing tensors have been partially understood so far, with two different systematic approaches via the prolongation procedure of the overdetermined system~\eqref{eq:defK} suggested in~\cite{EL} and via the study of top slot quadratic Killing tensors in~\cite[Theorem~3]{MN2}. The known results include the proof that quadratic Killing tensors on the classical compact groups with bi-invariant metric are decomposable~\cite{MNN}, as also are the quadratic Killing tensors on the symmetric space $\SU(6)/\Sp(3)$ \cite[Theorem~18]{EL}, while on the symmetric space $\mathrm{E}_6/\Ff$, there exists a $78$-dimensional $\mathrm{E}_6$-module of quadratic Killing tensors (``hidden symmetries'') which complements the module of the decomposable ones.

In the present paper, we fully classify quadratic Killing tensor fields on the real Grassmannians and the spaces $\SL(n)/\SO(n)$ by proving that they are decomposable.

\begin{theorem} \label{t:Gpq}
  Any quadratic Killing tensor field on the real Grassmannian $G(p,p+q) = \SO(p+q)/\SO(p) \times \SO(q), \; p, q \ge 1$, is decomposable.
\end{theorem}

\begin{theorem} \label{t:SLSO}
  Any quadratic Killing tensor field on the symmetric space $\SL(n)/\SO(n),$ $n \ge 1$, is decomposable.
\end{theorem}

Note that by~\cite[Theorem~2]{MN2} these theorems also imply decomposability of quadratic Killing tensor fields on the spaces $\SU(n)/\SO(n)$ and on the noncompact duals of the Grassmannians. 

The authors thank Vladimir Matveev for useful discussions.

\section{Preliminaries}
\label{s:prel}

\subsection{Symmetric spaces}
\label{ss:ss}

Standard references on geometry of Riemannian symmetric spaces are~\cite{Hel} and~\cite{Wolf}. Here we recall some necessary facts and introduce some notation. Let $M=G/H$, be an irreducible, simply connected, Riemannian globally symmetric space, where $G$ is the identity component of the full isometry group and $H \subset G$ is the isotropy subgroup. At the level of Lie algebras, one has the Cartan decomposition $\g = \h \oplus \m$, orthogonal relative to the Killing form $B$ on $\g$, where $\g$ and $\h$ are the Lie algebras of $G$ and $H$ respectively, and $\m$ is an $\h$-module which can be identified with the tangent space $T_oM$ to $M$ at the point $o = eH$. One has $[\m,\m], [\h,\h] \subset \h$ and $[\h,\m] \subset \m$. The inner products on $\h$ and on $\m$ are defined by $\|A\|^2=-\mu B(A,A)$ and $\|X\|^2= \ve\mu B(X,X)$, for $A \in \h, \, X \in \m$, where $\mu > 0$ and $\ve = 1$ (respectively, $\ve = -1$) if $M$ is of noncompact (respectively, of compact) type. Then for $A \in \h$ and $X, Y \in \m$ we have $\<[A,X],Y\> = -\ve \<A,[X,Y]\>$. The curvature tensor $R$ of $M$ at the point $o$ is given by $R(X,Y)Z = -[[X,Y],Z]$, for $X, Y, Z \in \m$.

A subspace $\s \subset \m$ is called a Lie triple system if $[[\s,\s],\s] \subset \s$. A connected submanifold $M' \subset M$ passing through $o$ is totally geodesic if and only if it is a domain on the exponent of a Lie triple system $T_oM' \subset \m$. Any totally geodesic submanifold of a symmetric space is by itself (locally) symmetric. Note that the restriction of a Killing tensor field on $M$ to a totally geodesic $M' \subset M$ gives a Killing tensor field on $M'$ (as the geodesics of $M'$ are geodesics of $M$).

A maximal abelian subspace $\ag \subset \m$ is called a Cartan subspace. All Cartan subspaces are conjugate under the action of the isotropy group $H$ on $\m$. A Cartan subspace is trivially a Lie triple system; its exponent (when $M$ is compact) is a flat totally geodesic torus in $M$. The dimension of $\ag$ is called the rank of the symmetric space $M$.

\subsection{Decomposability}
\label{ss:MN}

We will use the following necessary and sufficient condition of decomposability of quadratic Killing tensor fields~\cite[Theorem~3, Remark~2]{MN2}:
	
\begin{proposition} \label{p:topslot2}
		Let $M= G/H$ be an irreducible, simply connected globally symmetric space. Denote by $\sT(\m)$ the linear space of (constant) tensors $K$ of type $(0,4)$ on $\m$ which satisfy the following equations, for all $X, Y, Z, P \in \m$:
		\begin{gather}
			K(X,Y,Z,P) = K(Y,X,Z,P) = K(X,Y,P,Z), \quad \sigma_{X,Y,Z} \big(K(X,Y,Z,P)\big)=0, \label{eq:K2c}\\
			K(X,X,P,R(X,P)P)=K(P,P,X,R(P,X)X), \label{eq:K21}\\
			K(X,X,R(X,P)P,R(X,P)P)=K(P,P,R(P,X)X,R(P,X)X), \label{eq:K22}
		\end{gather}
		where $\sigma_{X,Y,Z}$ denotes the sum over cyclic permutations of $(X,Y,Z)$.
		
		Every quadratic Killing tensor field on $M$ is decomposable if and only if every element $K \in \sT(\m)$ is decomposable in the sense that there exists a symmetric operator $S: \h \to \h$ such that
  \begin{equation}\label{eq:K4}
    K(X,Y,Z,P) = \<S([Y,Z]), [X,P]\> + \<S([Y,P]), [X,Z]\>.
  \end{equation}
\end{proposition}
		
It is easy to see that the tensors $K$ defined by~\eqref{eq:K4} indeed satisfy~\eqref{eq:K2c}, \eqref{eq:K21} and~\eqref{eq:K22} using the fact that for a symmetric space, $R(X,Y)Z = -[[X,Y],Z]$ and the Jacobi identity.

We will also need the following fact which for the compact spaces is proved in~\cite[Lemma~3]{MNN}, and for the noncompact ones easily follows by duality.
	\begin{lemma}[{}] \label{l:threecommute}
		Let $M$ be an irreducible globally symmetric space. Then for any $K \in \sT(\m)$, we have $K(X_1,X_2,X_3,X_4) = 0$ provided $[X_1,X_3]=[X_2,X_3] =0$.
	\end{lemma}
Note that by the algebraic symmetries~\eqref{eq:K2c}, the same claim is true for any quadruple $(X_1,X_2,X_3,X_4)$ in which one of the elements from either the first or the second pair commutes with both elements of the other pair.

\subsection{Scheme of the proofs}
\label{ss:scheme}

The proofs of Theorems~\ref{t:Gpq} and~\ref{t:SLSO} follow similar scheme (which is different to the proof in~\cite{MNN} based on the induction by dimension). Suppose we are given a tensor $K \in \sT(\m)$.

First, starting from Lemma~\ref{l:threecommute} we show that if a trilinear functional $F: \m \times \m \times \m \to \br$ is symmetric by the last two arguments and is such that $F(X_1,X_3,X_4) = 0$ when $[X_1, X_3] = [X_1, X_4] = 0$, then there exists a linear map $\Phi: \m \to \h$ such that $F(X_1,X_3,X_4) = \<\Phi X_4, [X_1, X_3]\> + \<\Phi X_3, [X_1, X_4]\>$. We prove this in Proposition~\ref{p:FXYZ} and in Proposition~\ref{p:SLFXYZ}. It follows that for any $K \in \sT(\m)$, there exists a bilinear map $\Psi: \m \times \m \to \h$ such that $K(X_1,X_2,X_3,X_4) = \<\Psi(X_2,X_3), [X_1,X_4]\> + \<\Psi(X_2,X_4), [X_1,X_3]\>$. Using the symmetry of $K$ by the first two arguments, we then establish that $K$ has form~\eqref{eq:K4}, as required.

\section{Grassmannians and the Proof of Theorem~\ref{t:Gpq}}
\label{s:proofGpq}

\subsection{The real Grassmannians \texorpdfstring{$G(p,p+q) = \SO(p+q)/\SO(p) \times \SO(q)$}{G(p,p+q) = SO(p+q)/SO(p) \unichar{"00D7} SO(q)}}
\label{ss:Gpq}

We can assume that $p, q \ge 2$ and $p+q \ge 5$, as the claim for the round sphere and for the product of round spheres follows by the results from the Introduction.

We introduce the inner product on the Lie algebra $\g = \so(p+q)$ by $\<A_1,A_2\> = -\frac12 \Tr (A_1A_2)$, for $A_1,A_2 \in \g$. The summands of the Cartan decomposition $\g = \h \oplus \m$ are orthogonal and we have an orthogonal decomposition $\h = \so(p) \oplus \so(q)$. Moreover, $\<[A,X],Y\> = \<A,[X,Y]\>$, for all $A \in \h$ and $X, Y \in \m$.

The inner product on $\br^{p+q}$ will be also denoted $\ip$ (the meaning will be always clear from the context). For $a, b \in \br^{p+q}$ we denote $a \wedge b = ab^t - ba^t \in \so(p+q)$. Elements of $\g$ (and of $\m$) of this form are called \emph{simple}. We write $x, y, z, \dots$ for elements of $\br^p$, and $u,v, w, \dots$ for elements of $\br^q$. We have $\m = \br^p \wedge \br^q = \Span(x \wedge u \, | \, x \in \br^p, \, u \in \br^q)$ (where, to avoid ambiguity, we treat all vectors and subspaces as lying in $\br^{p+q}$).

For simple elements $X=x \wedge u, \, Y = y \wedge v \in \m$, where $x,y \in \br^p, \, u,v \in \br^q$, and for $A = A_1 \oplus A_2 \in \so(p) \oplus \so(q) = \h$ we have
\begin{gather*}
  \<X,Y\> = \<x,y\>\<u,v\>, \qquad [X,Y] = \<u,v\> y \wedge x + \<x,y\> v \wedge u \in \h, \\
  \<A, [X, Y]\> = \<u,v\> \<A_1 x, y\> + \<x,y\> \<A_2 u, v\>, \qquad [A,X] = (A_1x) \wedge u + x \wedge (A_2u) \in \m.
\end{gather*}
In particular, $[x \wedge u,y \wedge v] = 0$ if and only if either $x \perp y$ and $u \perp v$, or $x \wedge u$ and $y \wedge v$ are linearly dependent.

\begin{lemma} \label{l:rk1all}
  Suppose a polynomial $\phi(X, Y), \; X, Y \in \m$, is linear in $X$ and has the property that $\phi(X,Y) = 0$ whenever $[X,Y] = 0$ and $X$ is simple. Then $\phi(X,Y) = 0$ for all $X,Y \in \m$ with $[X,Y] = 0$.
\end{lemma}
\begin{proof}
  Let $X, Y \in \m$ commute. Using the polar decomposition (or equivalently, acting by the isotropy group on $\m$) we can find subspaces $\br^m_1 \subset \br^p$ and $\br^m_2 \subset \br^q$ and orthonormal bases $e_1, \dots, e_m$ for $\br^m_1$ and $f_1, \dots, f_m$ for $\br^m_2$, such that $X = \sum_{i=1}^{m} \la_i e_i \wedge f_i$ and $Y = \sum_{i=1}^{m} \mu_i e_i \wedge f_i + Y'$, where $Y' \in (\br^m_1)^\perp \wedge (\br^m_2)^\perp$, for some $\la_i \ne 0$ and some $\mu_i \in \br$. But then for every $i=1, \dots, m$, we have $[e_i \wedge f_i, Y] = 0$, and the claim follows from linearity of $\phi$ by $X$.
\end{proof}

\subsection{Proof of Theorem~\ref{t:Gpq}}
\label{ss:proofGpq}

We adopt the setup and notation of Section~\ref{s:prel} and Subsection~\ref{ss:Gpq}. We start with the following proposition (the first assertion should be well known in general, but we provide a short, explicit proof in our case for the sake of completeness).

\begin{proposition} \label{p:FXYZ}
{\ }
  \begin{enumerate}[label=\emph{(\alph*)},ref=\alph*]
  \item \label{it:F2}
  Suppose $F: \m \times \m \to \br$ is a bilinear map such that $F(X,Y) = 0$ when $[X,Y] = 0$. Then there exists a linear map $f: \h \to \br$ such that $F(X,Y) = f([X,Y])$.

  \item \label{it:F3}
  Suppose $F: \m \times \m \times \m \to \br$ is a trilinear map such that $F(X,Y,Z) = F(X,Z,Y)$, for all $X, Y, Z \in \m$, and $F(X,Y,Z) = 0$, for all $X, Y, Z \in \m$ such that $[X, Y] = [X, Z] = 0$. Then there exists a linear map $\Phi: \m \to \h$ such that $F(X,Y,Z) = \<\Phi Z, [X, Y]\> + \<\Phi Y, [X, Z]\>$.
\end{enumerate}
\end{proposition}
\begin{proof}
  \eqref{it:F2} Take $X = x \wedge u,\, Y = y \wedge v$, with $x, y \in \br^p, \, u,v \in \br^q$. The condition $[X, Y] = 0$ is equivalent to the fact that either $\<x,y\> = \<u,v\> = 0$ or $X$ and $Y$ are linearly dependent. For $u, v \in \br^q$, define the linear map $\Phi(u,v)$ on $\br^p$ by $\<\Phi(u,v)x,y\> = F(x \wedge u, y \wedge v)$. Then for $u \perp v$, the vector $\Phi(u,v)x$ must be a multiple of $x$, for all $x \in \br^p$. Choosing an orthonormal basis $f_\a$ for $\br^q$ and taking $u = a f_\a + b f_\b,\, v = -b f_\a + a f_\b$ for some $a,b \in \br$ and $\a \ne \b$ we obtain that $\Phi(f_\a,f_\b) = C_{\a\b} I_p$ and $\Phi(f_\a,f_\a) - \Phi(f_\b,f_\b) = C'_{\a\b} I_p$ for some $C_{\a\b}, C'_{\a\b} \in \br$. Then $\Phi(u,v) = \<Lu,v\> I_p + \<u,v\> A$, for some linear maps $L$ on $\br^q$ and $A$ on $\br^p$, and so $F(x \wedge u, y \wedge v) = \<Lu,v\> \<x,y\> + \<u,v\> \<Ax,y\>$. Moreover, we have $0 = F(x \wedge u, x \wedge u) = \<Lu,u\> \|x\|^2 + \|u\|^2 \<Ax,x\>$, for all $x \in \br^p, \, u \in \br^q$. It follows that $L = \la I_q + N$ and $A = -\la I_p + B$, for some $\la \in \br$ and $N \in \so(q), \, B \in \so(p)$. Taking $H = B \oplus N \in \so(p) \oplus \so(q) = \h$ we obtain that $F(X,Y) = \<H,[X,Y]\>$, for all simple $X, Y \in \m$. As simple elements span $\m$, the claim follows by bilinearity.

  \eqref{it:F3} 
  On several occasions throughout the proof, we will use the following simple fact. Suppose a trilinear form $f(x,y,z), \; x,y,z \in \br^n$ has the property that $f(x,y,z) = 0$ whenever $\<x,y\> = \<x,z\> = 0$. Then $f(x,y,z) = \<a,z\>\<x,y\> + \<b,y\>\<x,z\>$, for some $a,b \in \br^n$. Indeed, if $n=1$, there is nothing to prove. Otherwise, choosing an orthonormal basis for $\br^n$ we can write $f(x,y,z)= \sum_{ijk} f_{ijk} x_i y_j z_k$. For $i_1 \ne i_2$, take $y_{i_1}=z_{i_1}=-x_{i_2}, \, y_{i_2}=z_{i_2}=x_{i_1}$, and $x_i = 0$ when $i \ne i_1, i_2$. Then $\<x,y\> = \<x,z\> = 0$, and so we get $f(x,y,z)= 0$. Collecting the terms we find that $f_{ijk} = 0$ when $i \ne j,k$, that $f_{iij}=f_{iik}$ and $f_{iji}=f_{iki}$ when $i \ne j \ne k \ne i$, and that $f_{iii}=f_{iji}+ f_{iij}$ when $i \ne j$. Denoting $a_i = f_{jji}, \, b_i = f_{jij}$ for $j \ne i$, we obtain $f_{iii} = a_i + b_i$, and so $f(x,y,z) = \sum_{ijk} f_{ijk} x_i y_j z_k = \sum_{ij} (a_ix_jy_jz_i+b_ix_jy_iz_j) = \<a,z\>\<x,y\> + \<b,y\>\<x,z\>$, as claimed.

  
  Now given a map $F$ as in the assumption, take the elements $X, Y$ and $Z$ to be simple, so that $X = x \wedge u,\, Y = y \wedge v$ and $Z=z \wedge w$, with $x, y, z \in \br^p$ and $u,v,w \in \br^q$. Suppose $\<u,v\>=\<u,w\>=0$. Then $F(x \wedge u, y \wedge v, z \wedge w) = 0$ when $x \perp y, z$, and so from the argument in the previous paragraph, $F(x \wedge u, y \wedge v, z \wedge w) \equiv \<a(u,v,w),z\>\<x,y\> + \<b(u,v,w),y\>\<x,z\> \mod \cI$, for some trilinear maps $a, b: (\br^q)^3 \to \br^p$, where $\cI$ is the ideal generated by $\<u,v\>$ and $\<u,w\>$, and so $F(x \wedge u, y \wedge v, z \wedge w)$ belongs to the ideal generated by $\<x,y\>, \<x,z\>, \<u,v\>$ and $\<u,w\>$. In other words,
  \begin{multline}\label{eq:T14}
    F(x \wedge u, y \wedge v, z \wedge w) = T_1(z,u,v,w) \<x,y\> + T_2(y,u,v,w) \<x,z\> \\
    + T_3(w,x,y,z) \<u,v\> + T_4(v,x,y,z) \<u,w\>,
  \end{multline}
  for some quadrilinear forms $T_1,T_2,T_3$ and $T_4$.

  From the assumption we know that $F(x \wedge u, x \wedge u, z \wedge w) = 0$ when $\<x,z\>=\<u,w\>=0$, and so~\eqref{eq:T14} gives
  \begin{equation}\label{eq:T13}
  T_1(z,u,u,w) \|x\|^2 + T_3(w,x,x,z) \|u\|^2 =0, \quad \text{when } \<x,z\>=\<u,w\>=0.
  \end{equation}
  Take $z = e_i$ in~\eqref{eq:T13} and choose and fix an arbitrary nonzero $x \perp e_i$. We obtain that the polynomial $P(u,w)=T_1(e_i,u,u,w) + \|x\|^{-2}T_3(w,x,x,e_i) \|u\|^2$ vanishes when $w \perp u$. Choosing a basis for $\br^q$ we can write $P(u,w) = \sum_\a P_\a(u)w_\a$, where $P_\a(u)$ are quadratic forms on $\br^q$. It follows that the vector $(P_1(u), \dots, P_q(u))^t$ is proportional to $u$, for all $u \in \br^q$, which implies that $P_\a(u) u_\b = P_\b(u) u_\a$, and hence there exists a vector $a_i \in \br^q$ such that $P_\a(u) = \<a_i,u\> u_\a$. Then $P(u,w) = \<a_i,u\> \<u,w\>$, and so $T_1(e_i,u,u,w)= \<b_i,w\>\|u\|^2+\<a_i,u\> \<u,w\>$, where $b_i \in \br^q$ is defined by $\<b_i,w\> = -\|x\|^{-2}T_3(w,x,x,e_i)$ (note that by~\eqref{eq:T14}, the right-hand side is independent of the choice of $x \perp e_i$). Hence there exist linear maps $A_1, B_1: \br^p \to \br^q$ such that $T_1(z,u,u,w)= \<A_1z,u\> \<u,w\>+\<B_1z,w\>\|u\|^2$. By symmetry of~\eqref{eq:T13}, we similarly obtain that there exist linear maps $A_3, B_3: \br^q \to \br^p$ such that $T_3(w,x,x,z)= \<A_3w,x\> \<x,z\>+\<B_3w,z\>\|x\|^2$. Moreover, from~\eqref{eq:T13} we find that $B_3= -B_1^t$, so that $T_3(w,x,x,z)= \<A_3w,x\> \<x,z\>-\<B_1z,w\>\|x\|^2$.

  By a similar argument, the fact that $F(x \wedge u, y \wedge v, x \wedge u) = 0$ when $\<x,y\>=\<u,v\>=0$ implies the existence of linear maps $A_2, B_2: \br^p \to \br^q$ and $A_4: \br^q \to \br^p$ such that $T_2(y,u,v,u)= \<A_2y,u\> \<u,v\>+\<B_2y,v\>\|u\|^2$ and $T_4(v,x,y,x)= \<A_4v,x\> \<x,y\>-\<B_2y,v\>\|x\|^2$.

  Furthermore, we must have $F(x \wedge u, x \wedge u, x \wedge u) = 0$, for all $x \in \br^p$ and $u \in \br^q$, and so~\eqref{eq:T14} gives $(T_1(x,u,u,u) + T_2(x,u,u,u)) \|x\|^2 + (T_3(u,x,x,x) + T_4(u,x,x,x)) \|u\|^2 = 0$. Substituting the above expressions for $T_1,T_2,T_3$ and $T_4$, we obtain $\<A_1x, u\> + \<A_2 x, u\> + \<A_3u, x\> + \<A_4u, x\> = 0$, or equivalently,
  \begin{equation}\label{eq:A14}
    A_1+A_2+A_3^t+A_4^t=0.
  \end{equation}

  We now polarise the above expressions for $T_1,T_2,T_3$ and $T_4$ to obtain
  \begin{align*}
     T_1(z,u,v,w) & = \tfrac12(\<A_1z,u\> \<v,w\> + \<A_1z,v\> \<u,w\>) + \<K_1(z,w)u,v\> + \<B_1z,w\>\<u,v\>, \\
     T_2(y,u,v,w) & = \tfrac12(\<A_2y,u\> \<w,v\> + \<A_2y,w\> \<u,v\>) + \<K_2(y,v)u,w\> + \<B_2y,v\>\<u,w\>, \\
     T_3(w,x,y,z) & = \tfrac12(\<A_3w,x\> \<y,z\> + \<A_3w,y\> \<x,z\>) + \<K_3(z,w)x,y\> - \<B_1z,w\>\<x,y\>,\\
     T_4(v,x,y,z) & = \tfrac12(\<A_4v,x\> \<z,y\> + \<A_4v,z\> \<x,y\>) + \<K_4(v,y)x,z\>-\<B_2y,v\>\<x,z\>,
  \end{align*}
  for some bilinear maps $K_1, K_2: \br^p \times \br^q \to \so(q)$ and $K_3, K_4: \br^p \times \br^q \to \so(p)$. Substituting this into~\eqref{eq:T14} we obtain
  \begin{multline*}
    F(x \wedge u, y \wedge v, z \wedge w) = \<x,y\>\<u,w\>\<(A_1+A_4^t)z,v\> + \<x,z\>\<u,v\>\<(A_2+A_3^t)y,w\> + \\
    \<x,y\>\<L_1(z,w)u,v\> + \<x,z\>\<L_2(y,v)u,w\> + \<u,v\>\<L_3(z,w)x,y\> + \<u,w\>\<L_4(v,y)x,z\>,
  \end{multline*}
  where $L_1(z,w)=K_1(z,w)-\tfrac12 (A_1z) \wedge w, \, L_3(z,w)= K_3(z,w)-\tfrac12 (A_3w) \wedge z, \, L_2(y,v) = K_2(y,v) - \tfrac12 (A_2y) \wedge v$ and $L_4(v,y) = K_4(v,y) - \tfrac12 (A_4v) \wedge y$. Note that $L_1(z,w), L_2(y,v) \in \so(q)$ and $L_3(z,w),L_4(v,y) \in \so(p)$.

  Furthermore, by assumption we have $F(X,Y,Z) = F(X,Z,Y)$, and so $F(X,Y,Z) = \tfrac12 (F(X,Y,Z) + F(X,Z,Y))$. Substituting the above expression for $F(x \wedge u, y \wedge v, z \wedge w)$ and taking into account~\eqref{eq:A14} we obtain
  \begin{align*}
    F(x \wedge u, y \wedge v, z \wedge w) &= \<x,y\>\<N_1(z,w)u,v\> + \<x,z\>\<N_1(y,v)u,w\> \\
    &+ \<u,v\>\<N_2(z,w)x,y\>+ \<u,w\>\<N_2(v,y)x,z\> \\
    &= \<\Phi(z \wedge w),[x \wedge u, y \wedge v]\> + \<\Phi(y \wedge v),[x \wedge u, z \wedge w]\>,
  \end{align*}
  where $N_1=L_1+L_2, \, N_2 = L_3 + L_4$ and $\Phi(z \wedge w) = N_2(z,w) \oplus N_1(z,w), \Phi(y \wedge v) = N_2(y,v) \oplus N_1(y,v) \in \h$. Since $(z,w) \mapsto N_2(z,w)\oplus N_1(z,w)$ is bilinear and $\m \cong \br^p \otimes \br^q$ via $z\otimes w\mapsto z\wedge w$, it induces a unique linear map $\Phi:\m \to \h$. The claim now follows by trilinearity of $F$.
\end{proof}

Let $K \in \sT(\m)$. By Lemma~\ref{l:threecommute} and Proposition~\ref{p:FXYZ}, there exists a bilinear map $\Psi: \m \times \m \to \h$ such that
\begin{equation} \label{eq:KPsi}
  K(X_1,X_2,X_3,X_4) = \<\Psi(X_2,X_3), [X_1,X_4]\> + \<\Psi(X_2,X_4), [X_1,X_3]\>,
\end{equation}
for all $X_1,X_2,X_3,X_4 \in \m$. From the symmetries of $K$ we know that $K(X_1,X_2,X_3,X_4) = K(X_4,X_3,X_2,X_1)$. Substituting the expression on the right-hand side of~\eqref{eq:KPsi} and taking $[X_1,X_3] = 0$ in the resulting equation we obtain $\<\Psi(X_2,X_3)+\Psi(X_3,X_2), [X_1,X_4]\> + \<\Psi(X_3,X_1), [X_2,X_4]\>  = 0$, whenever $[X_1, X_3] = 0$, which gives
\begin{equation}\label{eq:Psi23}
  [\Psi(X_2,X_3)+\Psi(X_3,X_2), X_1] + [\Psi(X_3,X_1), X_2]  = 0, \quad \text{when } [X_1,X_3] = 0.
\end{equation}

Denote $\Psi(X_2,X_3)+\Psi(X_3,X_2) = A \oplus L$, where $A = A(X_2,X_3) \in \so(q)$ and $L = L(X_2,X_3) \in \so(p)$, and similarly $\Psi(X_3,X_1) = B \oplus N$, where $B = B(X_1,X_3) \in \so(q)$ and $N = N(X_1,X_3) \in \so(p)$. Then, taking $X_1 = x \wedge u$ and $X_2 = y \wedge v$, with $x, y \in \br^p, \, u,v \in \br^q$, in~\eqref{eq:Psi23} we obtain the following matrix equation:
\begin{equation}\label{eq:Psi23mat}
  (A(y \wedge v,X_3)u) x^t + u (L(y \wedge v,X_3)x)^t + (B(x \wedge u,X_3)v) y^t + v (N(x \wedge u,X_3)y)^t = 0,
\end{equation}
when $[x \wedge u, X_3] = 0$. Acting by both sides of~\eqref{eq:Psi23mat} on $x$ and assuming $x \perp y$ we get $\|x\|^2 A(y \wedge v,X_3)u + \<N(x \wedge u,X_3)y,x\> v = 0$. Taking the inner product with $u$ we obtain $\<N(x \wedge u,X_3)y,x\> \<v,u\> = 0$, and so $\<N(x \wedge u,X_3)y,x\> = 0$ when $x \perp y$ and $[x \wedge u, X_3] = 0$. As $N(x \wedge u,X_3)$ is skew-symmetric, we obtain $N(x \wedge u,X_3)x =0$, when $[x \wedge u, X_3] = 0$. Then the preceding equation $\|x\|^2 A(y \wedge v,X_3)u + \<N(x \wedge u,X_3)y,x\> v = 0$ implies that $A(y \wedge v,X_3)u = 0$ when $[x \wedge u, X_3] = 0$. Substituting this into~\eqref{eq:Psi23mat} and acting by both sides on $x$ we obtain $\<x, y\> B(x \wedge u,X_3)v = 0$, and so $B(x \wedge u,X_3) = 0$. Then from~\eqref{eq:Psi23mat} we get $u (L(y \wedge v,X_3)x)^t + v (N(x \wedge u,X_3)y)^t = 0$, and so taking $u, v \in \br^q$ linearly independent we find $L(y \wedge v,X_3)x = N(x \wedge u,X_3)y = 0$ for all $x, y \in \br^p, \, u,v \in \br^q$ and $X_3 \in \m$ such that $[x \wedge u, X_3] = 0$. It follows that $N(x \wedge u,X_3) = 0$. Hence $\Psi (X_3,X_1) = 0$, for all $X_1,X_3 \in \m$ such that $[X_1,X_3] = 0$ and $X_1$ is simple.

From Lemma~\ref{l:rk1all} we obtain that $\Psi (X_3,X_1) = 0$, for all $X_1,X_3 \in \m$ such that $[X_1,X_3] = 0$. Then by Proposition~\ref{p:FXYZ}\eqref{it:F2}, there exists a linear map $\psi: \h \to \h$ such that $\Psi(X,Y) = \psi([X,Y])$, for all $X, Y \in \m$. Then from equation~\eqref{eq:KPsi} we obtain $K(X_1,X_2,X_3,X_4) = \<\psi([X_2,X_3]), [X_1,X_4]\> + \<\psi([X_2,X_4]), [X_1,X_3]\>$. But now from the symmetries of $K$ we get $K(X_1,X_2,X_3,X_4) = \frac12(K(X_1,X_2,X_3,X_4) + K(X_2,X_1,X_3,X_4))$, and it follows that $K(X_1,X_2,X_3,X_4) = \<S([X_2,X_3]), [X_1,X_4]\> + \<S([X_2,X_4]), [X_1,X_3]\>$, where $S =\frac12(\psi + \psi^t)$ is a symmetric operator on $\h$, as required by~\eqref{eq:K4}.

This completes the proof of Theorem~\ref{t:Gpq}.

\section{Symmetric space \texorpdfstring{$\GL(n)/\SO(n)$}{GL(n)/SO(n)} and the Proof of Theorem~\ref{t:SLSO}}
\label{s:proofSLSO}

\subsection{Symmetric space \texorpdfstring{$\boldsymbol{\GL(n)/\SO(n)}$}{GL(n)/SO(n)}}
\label{ss:GLSO}

It will be easier to work with the reducible symmetric space $\GL(n)/\SO(n) = \br \times \SL(n)/\SO(n)$ with the product metric. Note that by~\cite[Theorem~1]{MN2} the algebra of Killing tensor fields on $\GL(n)/\SO(n)$ is the graded product of the algebras of Killing tensor fields on $\SL(n)/\SO(n)$ and on $\br$ (the latter being the algebra of real polynomials in one variable).

For the symmetric space $\GL(n)/\SO(n)$, we have $\g=\gl(n),\, \h=\so(n)$, and in the Cartan decomposition $\g=\h \oplus \m$, the subspace $\m$ is the space of $n \times n$ real symmetric matrices. We have $\m = \br I_n \oplus \m_0$, where $\m_0$ is the subspace of matrices from $\m$ with trace zero. If we can prove that any tensor $K \in \sT(\m)$ has the form~\eqref{eq:K4} for a symmetric operator $S: \h \to \h$, the claim of Theorem~\ref{t:SLSO} will be established. Indeed, given any tensor $K_0 \in \sT(\m_0)$, we define the tensor $K \in \sT(\m)$ by $K(I_n, \cdot,\cdot,\cdot) = 0$ and $K(X, Y, Z, P) = K_0(X, Y, Z, P)$, for $X, Y, Z, P \in \m_0$. As the curvature tensor of $\GL(n)/\SO(n)$ satisfies $R(I_n,X)Y = R(X,Y)I_n = 0$, for any $X, Y \in \m_0$, the tensor $K$ belongs to $\sT(\m)$, and hence has the form~\eqref{eq:K4}, as also does the original tensor $K_0$, if we restrict~\eqref{eq:K4} to $X, Y, Z, P \in \m_0$. Then the claim for $\SL(n)/\SO(n)$ follows from Proposition~\ref{p:topslot2}.

So for the rest of the proof, we will work with the symmetric space $\GL(n)/\SO(n)$ and with tensors $K \in \sT(\m)$. Note that $\GL(2)/\SO(2)$ is isometric to the product of the line and the hyperbolic plane, and so any Killing tensor field on it is decomposable by the results from the Introduction. For $n=3,4$, the fact that any tensor $K \in \sT(\m)$ has the form~\eqref{eq:K4} is proved by a direct computation: the space of tensors $K$ satisfying~\eqref{eq:K2c} has dimension $\frac{1}{12} N^2(N^2-1)$, where $N = \dim \m$~\cite[Remark~3]{MN2}, which gives $105$ for $n=3$ and $825$ for $n=4$. Then the linear system of equations for the components of $K$ given by~\eqref{eq:K21} and~\eqref{eq:K22} is easily handled by computer algebra.

For the rest of the proof we assume that $n \ge 5$.

We introduce the inner product on the Lie algebra $\g = \gl(n)$ by $\<A_1,A_2\> = \frac12 \Tr (A_1A_2^t)$, for $A_1,A_2 \in \g$. Then the summands of the Cartan decomposition $\g = \h \oplus \m$ are orthogonal; we denote $\ip$ the restriction of the above inner product $\ip$ on $\g$ to each of them.

For $x, y \in \br^n$, we define $x \odot y = xy^t + yx^t$. Then $\m = \br^n \cdot \br^n$, the space of $n \times n$ real symmetric matrices (and $\h = \so(n)$). We call an element $X \in \m$ \emph{simple}, if $\rk X \le 1$, that is if $X = \pm x \odot x$ for some $x \in \br^n$. Note that $\<L,[X,Y]\>=\<[L,Y],X\>$ for $L \in \h = \so(n)$ and $X,Y \in \m$. For simple elements $X=x \odot x, \, Y = y \odot y \in \m$, where $x,y \in \br^n$, and for $A \in \so(n) = \h$ we have
\begin{gather*}
  \<X,Y\> = 2 \<x,y\>^2, \qquad [X,Y] = 4\<x,y\> \, x \wedge y \in \h, \\
  \<A, [X, Y]\> = 4\<x,y\> \<A y, x\>, \qquad [A,X] = 2(Ax) \odot x \in \m.
\end{gather*}
In particular, $[x \odot x,y \odot y] = 0$ if and only if either $x \perp y$, or $x$ and $y$ are linearly dependent.

We will use the following fact.

\begin{lemma} \label{l:rk1sym}
  Suppose a polynomial $\phi(X, Y), \; X, Y \in \m$, is linear in $X$ and has the property that $\phi(X,Y) = 0$ whenever $[X,Y] = 0$ and $X$ is simple. Then $\phi(X,Y) = 0$ for all $X,Y \in \m$ with $[X,Y] = 0$.
\end{lemma}
\begin{proof}
  Suppose $X, Y \in \m$ commute. Using the canonical form, we can find a subspaces $\br^m \subset \br^n$ and an orthonormal basis $e_1, \dots, e_m$ for $\br^m$ such that $X = \sum_{i=1}^{m} \la_i e_i \odot e_i$ and $Y = \sum_{i=1}^{m} \mu_i e_i \odot e_i + Y'$, where $Y' \in (\br^m)^\perp \odot (\br^m)^\perp$, for some $\la_i \ne 0$ and some $\mu_i \in \br$. But then for every $i=1, \dots, m$, we have $[e_i \odot e_i, Y] = 0$, and the claim follows from linearity of $\phi$ by $X$.
\end{proof}

\subsection{Proof of Theorem~\ref{t:SLSO}}
\label{ss:proofSLSO}

The following lemma can be deduced from Lemma~\ref{l:threecommute} (in which the space is required to be irreducible), but it is easier to give a direct, independent proof in our case.

\begin{lemma} \label{l:SLK1340}
  If $K \in \sT(\m)$, then $K(X_1,X_2,X_3,X_4) = 0$ for all $X_1,X_2,X_3,X_4 \in \m$ such that $[X_1,X_3]=[X_1,X_4]=0$.
\end{lemma}
\begin{proof}
  The tensor $K$ is symmetric by its last two arguments by~\eqref{eq:K2c}, and so it is sufficient to show that $K(X_1,X_2,X_3,X_3) = 0$ for any $X_1,X_2,X_3 \in \m$ with $[X_1,X_3]=0$. Furthermore, by Lemma~\ref{l:rk1sym} it is sufficient to prove this fact under an additional assumption that $X_1$ is simple, and so it is sufficient to show that $K(X_1,X_2,X_3,X_4) = 0$ when $[X_1,X_3]=[X_1,X_4]=0$ and $X_1=x \odot x$ is simple. Moreover, as the subspace of elements of $\m$ commuting with $x \odot x, \, x \ne 0$, is spanned by $x \odot x$ and $x^\perp \odot x^\perp$, it suffices to prove that $K(X_1,X_2,X_3,X_4) = 0$ when $X_3=u \odot u, \, X_4=v \odot v$, with $u, v \in (\br \, x) \cup (x^\perp)$ and $X_2=w \odot w$, with an arbitrary $w \in \br^n$. We denote $V = \Span(x,u,v,w)$ and $\m' = V \odot V \subset \m$. The subspace $\m' \subset \m$ is a Lie triple system, and so the restriction $K'$ of the tensor $K$ to $\m'$ still satisfies the equations~\eqref{eq:K2c}, \eqref{eq:K21} and~\eqref{eq:K22}, and $K(X_1,X_2,X_3,X_4) = K'(X_1,X_2,X_3,X_4)$. As for $n \le 4$ the tensor $K'$ has the form given in~\eqref{eq:K4}, we obtain that $K'(X_1,X_2,X_3,X_4) = 0$ when $[X_1,X_3]=[X_1,X_4]=0$, which completes the proof.
\end{proof}

The next step in the proof is the following.
\begin{proposition} \label{p:SLFXYZ}
Assume $n \ge 3$.
  \begin{enumerate}[label=\emph{(\alph*)},ref=\alph*]
  \item \label{it:SLF2}
  Suppose $F: \m \times \m \to \br$ is a bilinear map such that $F(X,Y) = 0$ when $[X,Y] = 0$. Then there exists a linear map $f: \h \to \br$ such that $F(X,Y) = f([X,Y])$.

  \item \label{it:SLF3}
  Suppose $F: \m \times \m \times \m \to \br$ is a trilinear map such that $F(X,Y,Z) = F(X,Z,Y)$, for all $X, Y, Z \in \m$, and $F(X,Y,Z) = 0$, for all $X, Y, Z \in \m$ such that $[X, Y] = [X, Z] = 0$. Then there exists a linear map $\Phi: \m \to \h$ such that $F(X,Y,Z) = \<\Phi Z, [X, Y]\> + \<\Phi Y, [X, Z]\>$.
\end{enumerate}
\end{proposition}
\begin{proof}
  \eqref{it:SLF2} From the assumption, the polynomial $P(x,y)=F(x \odot x, y \odot y)$ is zero when either $x \perp y$ or $x = y$. The polynomial $\<x,y\> \in \br[x,y]$ is irreducible when $n \ge 2$, and so the principal ideal $\cI \subset \br[x,y]$ generated by $\<x,y\>$ it is also irreducible. As $P(x,y) = 0$ when $\<x,y\> = 0$, we obtain that $P(x,y)$ is divisible $\<x,y\>$. Comparing the degrees we find that there is a linear operator $L$ on $\br^n$ such $P(x,y)=\<x,y\>\<Lx,y\>$. As $P(x,x)=0$ we obtain $L \in \so(n)$, and so $F(x \odot x, y \odot y) = -\frac14 \<L, [x \odot x, y \odot y]\>$, for all $x, y \in \br^n$, and the claim follows by bilinearity.

  \eqref{it:SLF3}
  From the assumption, the polynomial $P(x,y,z)=F(x \odot x, y \odot y, z \odot z)$ is zero when $y,z \in (\br \, x) \cup (x^\perp)$. We first show that the fact that $P(x,y,z)=0$ when $\<x,y\>=\<x,z\> = 0$ implies that $P$ belongs to the ideal of $\br[x,y,z]$ generated by the polynomials $\<x,y\>$ and $\<x,z\>$. We complexify everything and consider the polynomials $\<x,y\>=\sum_{i=1}^n x_iy_i,\, \<x,z\>=\sum_{i=1}^{n} x_iz_i$ and $P$ as the elements of $\bc[x,y,z]$. Note that $P$ is still zero on the variety in $\bc^{3n}$ defined by $\<x,y\>=\<x,z\> = 0$. Indeed, this variety contains all the points $(x,y,z) \in \bc^{3n}$ such that $x=(s,u)^t,\, y = (-\<u,v\>, sv)^t, \, z = (-\<u,w\>, sw)^t$, for all $(u,v,w,s) \in \br^{n-1} \oplus \br^{n-1} \oplus \br^{n-1} \oplus \br$, and hence all the points $(x,y,z) \in \bc^{3n}$ given by the same formulas with $(u,v,w,s) \in \bc^{n-1} \oplus \bc^{n-1} \oplus \bc^{n-1} \oplus \bc$. But the latter is the subset of $\bc^{3n}$ defined by the equations $\<x,y\>=\<x,z\> = 0$ and the inequality $x_1 \ne 0$. Its Zariski closure is the variety defined by the equations $\<x,y\>=\<x,z\> = 0$. Consider the principal ideal $\cI=\cI(\<x,y\>) \subset \bc[x,y,z]$ and the quotient ring $\mathcal{R} = \bc[x,y,z]/\cI$. As $n \ge 3$ by our assumption, the quadratic form $(x,y) \to \<x,y\>$ has rank at least $6$, and so the ring $\mathcal{R}$ is a \emph{unique factorisation domain} by~\cite{Nag}. What is more, the image $\pi\<x,z\>$ of the polynomial $\<x,z\>$ under the natural projection $\pi:\bc[x,y,z] \to \mathcal{R}$ is irreducible. Indeed, suppose there exist elements $P_1,P_2,P_3 \in \bc[x,y,z]$ such that we have a presentation $\<x,z\> = P_1P_2+\<x,y\>P_3$. Denote $m=\deg P_1 + \deg P_2$. If $m \ge 3$, then the product of the homogeneous components of $P_1$ and $P_2$ of the maximal degrees is divisible by $\<x,y\>$, and so one of these components, say the one from $P_1$, would be divisible by $\<x,y\>$. Then we can replace $P_1$ by a polynomial $P_1'$ of lower degree, such that $\pi(P_1)=\pi(P_1')$ and so $m$ becomes smaller. Repeating this process we obtain a presentation $\<x,z\> = Q_1Q_2+\<x,y\>Q_3$ such that $\pi(Q_1)=\pi (P_1), \, \pi(Q_2)=\pi (P_2)$ and $\deg Q_1 + \deg Q_2 \le 2$. Clearly $Q_1,Q_2 \ne 0$, as $\<x,z\>$ is not divisible by $\<x,y\>$. Suppose neither of $Q_1,Q_2$ is a nonzero constant. Then the only possibility is $\deg Q_1 = \deg Q_2 = 1$. Comparing the degrees we obtain that the constant terms of both $Q_1$ and $Q_2$ are zeros, and so both $Q_1$ and $Q_2$ are linear forms on $\bc^{3n}$, and then $Q_3 = c \in \bc$. But as $n \ge 3$, the quadratic form $\<x,z\> - c \<x,y\>$ is not a product of two linear forms. Hence $\pi(\<x,z\>)$ is irreducible in $\mathcal{R}$. As $P(x,y,z)=0$ when $\<x,y\>=\<x,z\> = 0$, there is an $l \ge 1$ such that $P^l$ lies in the ideal of $\bc[x,y,z]$ generated by $\<x,y\>$ and $\<x,z\>$, and so $\pi(P)^l$ is divisible by $\pi(\<x,z\>)$ in $\mathcal{R}$. Since $\pi(\<x,z\>)$ is irreducible, we obtain that $\pi(P)$ itself is divisible by $\pi(\<x,z\>)$ in $\mathcal{R}$. Hence $P$ lies in the ideal of $\bc[x,y,z]$ generated by $\<x,y\>$ and $\<x,z\>$, and as all three polynomials are real, the same as true in $\br[x,y,z]$.

  Thus we obtain
  \begin{equation}\label{eq:SLPxyz}
  P(x,y,z)= \<x,y\> T_1(x,y,z) + \<x,z\> T_2(x,y,z),
  \end{equation}
  and comparing the degrees we can take $T_1$ (respectively, $T_2$) to be linear in $x$ and $y$ and quadratic in $z$ (respectively, linear in $x$ and $z$ and quadratic in $y$). We know that $P(x,x,z) = 0$ when $x \perp z$, which by~\eqref{eq:SLPxyz} gives $T_1(x,x,z)=0$ when $\<x,z\> = 0$. As $\<x,z\>$ is irreducible, the polynomial $T_1(x,x,z)$ is divisible by $\<x,z\>$, and comparing the degrees we find a linear operator $C_1$ on $\br^n$ such that $T_1(x,x,z)= \<x,z\>\<C_1x,z\>$, for all $x,z \in \br^n$. As $T_1$ is linear by the first two components, we obtain $T_1(x,y,z) = \<x,z\>\<C_1y,z\> + \<L_1(z)x,y\>$, where $L_1$ is a quadratic form on $\br^n$ with the values in $\so(n)$. Similarly, $T_2(x,y,z) = \<x,y\>\<C_2y,z\> + \<L_2(y)x,z\>$, for some linear operator $C_2$ on $\br^n$ and a quadratic map $L_2$ from $\br^n$ to $\so(n)$. Moreover, we also know that $P(x,x,x)=0$. Substituting into~\eqref{eq:SLPxyz} we find that $C=C_1 + C_2 \in \so(n)$. Then~\eqref{eq:SLPxyz} becomes $P(x,y,z)= \<x,y\> \<x,z\> \<Cy,z\> + \<x,y\> \<L_1(z)x,y\> + \<x,z\> \<L_2(y)x,z\>$. As by assumption, $P(x,y,z)=\frac12 (P(x,y,z) + P(x,z,y))$, and so $P(x,y,z) = \<x,y\> \<L(z)x,y\> + \<x,z\> \<L(y)x,z\>$, where $L=\frac12(L_1+L_2)$ is a quadratic form on $\br^n$ with the values in $\so(n)$. This can be written as $F(x \odot x, y \odot y, z \odot z) = -\frac14 \<L(z), [x \odot x, y \odot y]\> -\frac14 \<L(y), [x \odot x, z \odot z]\>$. But any quadratic form on $\br^n$ is the restriction of a linear form on $\m$ to simple vectors. It follows that for some linear map $\Phi: \m \to \h$, the required identity $F(X,Y,Z) = \<\Phi Z, [X, Y]\> + \<\Phi Y, [X, Z]\>$ holds for all simple elements $X,Y$ and $Z$, and the claim follows by trilinearity of $F$.
\end{proof}

\medskip

Now suppose that $K \in \sT(\m)$ (and $n \ge 3$). From Lemma~\ref{l:SLK1340}, $K(X_1,X_2,X_3,X_4) = 0$ whenever $[X_1,X_3]=[X_1,X_4]=0$. Then by Proposition~\ref{p:SLFXYZ}, there exists a bilinear map $\Psi: \m \times \m \to \h$ such that
\begin{equation} \label{eq:SLKPsi}
  K(X_1,X_2,X_3,X_4) = \<\Psi(X_2,X_3), [X_1,X_4]\> + \<\Psi(X_2,X_4), [X_1,X_3]\>,
\end{equation}
for all $X_1,X_2,X_3,X_4 \in \m$. From the symmetries~\eqref{eq:K2c} of $K$ we have $K(X_1,X_2,X_3,X_4) = K(X_4,X_3,X_2,X_1)$. Substituting into~\eqref{eq:SLKPsi} and then taking $[X_1,X_3] = 0$ in the resulting equation we obtain $\<\Psi(X_2,X_3)+\Psi(X_3,X_2), [X_1,X_4]\> + \<\Psi(X_3,X_1), [X_2,X_4]\>  = 0$, whenever $[X_1, X_3] = 0$, which gives
\begin{equation}\label{eq:SLPsi23}
  [\Psi(X_2,X_3)+\Psi(X_3,X_2), X_1] + [\Psi(X_3,X_1), X_2]  = 0, \quad \text{when } [X_1,X_3] = 0.
\end{equation}
Denote $A(X_2,X_3) = \Psi(X_2,X_3)+\Psi(X_3,X_2)$. Taking $X_1 = x \odot x$ and $X_2 = y \odot y$, with $x, y \in \br^n$, we obtain $A(y \odot y, X_3)x \odot x + \Psi(X_3, x \odot x)y \odot y = 0$, whenever $[x \odot x, X_3] = 0$. It follows that $\Psi(X_3, x \odot x)y \in \Span(x,y)$ when $x$ and $y$ are linearly independent (and $[x \odot x, X_3] = 0$). Fix $x \ne 0$ and $X_3$ which commutes with $x \odot x$. As $\Psi(X_3, x \odot x) \in \so(n)$, taking $y \perp x$ we obtain $\Psi(X_3, x \odot x) = x \wedge u(x, X_3)$, for some $u(x, X_3) \in x^\perp$. Since $n \ge 3$, we can take $y=x+v$, where $v \perp x, u$ and $v \ne 0$. Then $\Psi(X_3, x \odot x)y = -\|x\|^2 u(x, X_3)$, which lies in $\Span(x,x+v)$ only when $u(x, X_3)=0$. It follows that $\Psi(X_3, x \odot x) = 0$ when $[x \odot x, X_3] = 0$, and so by Lemma~\ref{l:rk1sym} and Proposition~\ref{p:SLFXYZ}\eqref{it:SLF2}, there exists a linear map $\psi: \h \to \h$ such that $\Psi(X,Y) = \psi([X,Y])$, for all $X, Y \in \m$. Then from~\eqref{eq:SLKPsi} we get $K(X_1,X_2,X_3,X_4) = \<\psi([X_2,X_3]), [X_1,X_4]\> + \<\psi([X_2,X_4]), [X_1,X_3]\>$. From~\eqref{eq:K2c} we obtain $K(X_1,X_2,X_3,X_4) = \frac12(K(X_1,X_2,X_3,X_4) + K(X_2,X_1,X_3,X_4))= \<S([X_2,X_3]), [X_1,X_4]\> + \<S([X_2,X_4]), [X_1,X_3]\>$, where $S =\frac12(\psi + \psi^t)$ is a symmetric operator on $\h$, as given by~\eqref{eq:K4}.

This completes the proof of Theorem~\ref{t:SLSO}.

\end{document}